\documentclass[final,nomarks]{dmtcs-episciences}
\usepackage[utf8]{inputenc}
\usepackage[T1]{fontenc}
\usepackage{amsmath,amsthm,amsfonts,amssymb,calc,microtype,mathtools,thmtools,underscore}
\usepackage[round]{natbib}
\usepackage[usenames,dvipsnames,svgnames,table]{xcolor}
\hypersetup{
	colorlinks =true,
	citecolor={black},
	linkcolor={black},
	pdftitle={Improved Product Structure for Graphs on Surfaces},
	pdfauthor={Distel--Hickingbotham--Huynh----Wood},
}

\newcommand{\defn}[1]{\textcolor{Maroon}{\emph{#1}}}

\usepackage[shortlabels]{enumitem}
\setlist[itemize]{topsep=0ex,itemsep=0ex,parsep=0ex}
\setlist[enumerate]{topsep=0ex,itemsep=0ex,parsep=0ex}

\usepackage[noabbrev,capitalise]{cleveref}
\crefname{lem}{Lemma}{Lemmas}
\crefname{thm}{Theorem}{Theorems}
\crefname{cor}{Corollary}{Corollaries}
\crefname{prop}{Proposition}{Propositions}
\crefname{conj}{Conjecture}{Conjectures}
\crefname{open}{Open Problem}{Open Problems}
\crefname{obs}{Observation}{Observations}

\crefformat{equation}{(#2#1#3)}
\Crefformat{equation}{Equation #2(#1)#3}

\theoremstyle{plain}
\newtheorem{thm}{Theorem}
\newtheorem{lem}[thm]{Lemma}

\newtheorem{obs}[thm]{Observation}

\theoremstyle{definition}

\DeclarePairedDelimiter{\floor}{\lfloor}{\rfloor}

\renewcommand{\le}{\leqslant}
\renewcommand{\leq}{\leqslant}

\renewcommand{\geq}{\geqslant}

\theoremstyle{definition}

\DeclareMathOperator*{\dist}{dist}
\DeclareMathOperator*{\bs}{\backslash}
\DeclareMathOperator*{\sse}{\subseteq}

\DeclareMathOperator*{\Pcal}{\mathcal{P}}

\DeclareMathOperator*{\Lcal}{\mathcal{L}}

\DeclareMathOperator{\tw}{tw}

\author{Marc Distel \and Robert Hickingbotham 	\and Tony Huynh 	\and David R. Wood}

\title{Improved Product Structure for Graphs on Surfaces\thanks{Research of M.D.\ and R.H.\ is supported by an Australian Government Research Training Program Scholarship. Research of T.H.\ and D.W.\ is supported by the Australian Research Council.}	
}

\affiliation{
	School of Mathematics, Monash University, Melbourne, Australia}
\received{2021-12-21}

\revised{2022-09-16, 2022-09-22}

\accepted{2022-09-26}

\begin{document}
	\publicationdetails{24}{2022}{2}{6}{8877}
	\maketitle
	\begin{abstract}
	 	Dujmovi\'c, Joret, Micek, Morin, Ueckerdt and Wood~[J.~ACM 2020] proved that for every graph $G$ with Euler genus $g$ there is a graph $H$ with treewidth at most 4 and a path $P$ such that $G\sse H \boxtimes P \boxtimes K_{\max\{2g,3\}}$. We improve this result by replacing ``4'' by ``3'' and with $H$ planar. We in fact prove a more general result in terms of so-called framed graphs. This implies that every $(g,d)$-map graph is contained in $ H \boxtimes P\boxtimes K_\ell$, for some planar graph $H$ with treewidth $3$, where $\ell=\max\{2g\floor{\frac{d}{2}},d+3\floor{\frac{d}{2}}-3\}$. It also implies that every $(g,1)$-planar graph (that is, graphs that can be drawn in a surface of Euler genus $g$ with at most one crossing per edge) is contained in $H\boxtimes P\boxtimes K_{\max\{4g,7\}}$, for some planar graph $H$ with treewidth $3$. 
	\end{abstract}
	
The motivation for this work is the following question: what is the global structure for graphs embeddable in a fixed surface? \citet{DJMMUW20} answered this question for planar graphs\footnote{A \defn{plane graph} is a graph embedded in the plane with no crossings. A \defn{plane triangulation} is a plane graph in which every face is bounded by a triangle (that is, has length 3). A \defn{plane near-triangulation} is a plane graph, where the outer-face is a  cycle, and every internal face is a triangle.} in terms of products\footnote{For two graphs $ G $ and $ H $, the \defn{strong product} $ G \boxtimes H $ is the graph with vertex-set $ V(G) \times V(H) $ and an edge between two vertices $ (v,w) $ and $ (v',w') $ if and only if
	$ v=v'$ and $ ww' \in E(H) $, or
	$ w=w' $ and $ vv' \in E(G) $, or
	$ vv'\in E(G) $ and $ ww' \in E(H) $.} of graphs of bounded treewidth\footnote{A \defn{tree-decomposition} of a graph $G$ is a collection $(B_x\subseteq V(G):x\in V(T))$ of subsets of $V(G)$ (called \defn{bags}) indexed by the nodes of a tree $T$, such that (a) for every edge $uv\in E(G)$, some bag $B_x$ contains both $u$ and $v$, and (b) for every vertex $v\in V(G)$, the set $\{x\in V(T):v\in B_x\}$ induces a non-empty subtree of $T$. The \defn{width} of a tree-decomposition is the size of a largest bag minus 1. The \defn{treewidth} of a graph $G$, denoted by \defn{$\tw(G)$}, is the minimum width of a tree-decomposition of $G$. 
	%These definitions are due to \citet{RS-II}. \tony{According to Wikipedia, treewidth was previously discovered by Umberto Bertelè and Francesco Brioschi (1972) under the name of dimension. It was later rediscovered by Rudolf Halin (1976), and then by Robertson and Seymour.}  \david{As I understand it, none of Bertelè -- Brioschi -- Halin defined a tree-decomposition.}
	Treewidth is recognised as the most important measure of how similar a given graph is to a tree.}\,\footnote{A graph $G$ is \defn{contained} in a graph $X$ if $G$ is isomorphic to a subgraph of $X$. A multigraph $G$ is \defn{contained} in a graph $X$ if the simple graph underlying $G$ is contained in $X$.}.

\begin{thm}[\citep{DJMMUW20}] 
	\label{PlanarStructure}
	Every planar graph is contained in $H\boxtimes P \boxtimes K_3$ for some planar graph $H$ with treewidth 3 and for some path $P$.
\end{thm}

This result, now known as the Planar Graph Product Structure Theorem, has been the key tool in solving several open problems regarding queue layouts~\citep{DJMMUW20}, non-repetitive colourings~\citep{DEJWW20}, centred colourings~\citep{DFMS21}, clustered colourings~\citep{DEMWW22}, adjacency labellings~\citep{BGP20,DEJGMM21,EJM}, vertex rankings~\citep{BDJM}, twin-width~\citep{BKW,BDHK} and infinite graphs~\citep{HMSTW}. 

\citet{DJMMUW20} generalised \cref{PlanarStructure} for graphs embeddable in any fixed surface\footnote{The \defn{Euler genus} of a surface with $h$ handles and $c$ cross-caps is $2h+c$. The \defn{Euler genus} of a graph $G$ is the minimum integer $g\geq 0$ such that there is an embedding of $G$ in a surface of Euler genus $g$; see \cite{MoharThom} for more about graph embeddings in surfaces. A \defn{triangulation} of  a surface $\Sigma$ is a graph embedded in $\Sigma$ with no crossings, such that every face is a triangle. } as follows. A graph $H$ is \defn{apex} if $H-v$ is planar for some vertex $v$ of $H$. 

\begin{thm}[\citep{DJMMUW20}] 
	\label{SurfaceStructure}
	Every graph with Euler genus $g$ is contained in $H\boxtimes P \boxtimes K_{\max\{2g,3\}}$ for some apex graph $H$ with treewidth 4 and for some path $P$.
\end{thm}

This paper improves this bound on the treewidth of $H$ from 4 to 3.

\begin{thm}
	\label{SurfaceStructureImproved}
	Every graph with Euler genus $g$ is contained in  $H\boxtimes P \boxtimes K_{\max\{2g,3\}}$ for some planar graph $H$ with treewidth 3 and for some path $P$.
\end{thm}

The bound on the treewidth of $H$  in \cref{SurfaceStructureImproved} is optimal since \citet{DJMMUW20} showed that for every integer $\ell\geq 0$ there is a planar graph $G$ such that if $G$ is contained in $H\boxtimes P\boxtimes K_\ell$, then $H$ has treewidth at least 3.

We in fact prove a more general result in terms of so-called framed graphs. Let $G$ be a multigraph embedded in a surface $\Sigma$ without crossings, where each face is bounded by a cycle. For any integer $d\geq 3$, let $G^{(d)}$ be the multigraph embedded in $\Sigma$ obtained from $G$ as follows: for each face $F$ of $G$ bounded by a cycle $C$ of length at most $d$, for all distinct non-adjacent vertices $v,w$ in $C$, add an edge $vw$ across $F$ to $G^{(d)}$.  We say that $G^{(d)}$ is a \defn{$(\Sigma,d)$-framed multigraph} with \defn{frame} $G$. If $\Sigma$ has Euler genus at most $g$, then $G^{(d)}$ is a \defn{$(g,d)$-framed multigraph}.

We prove the following theorem.

\begin{thm}
	\label{FramedProduct}
	For all integers $g\geq 0$ and $d\geq 3$, every $(g,d)$-framed multigraph is contained in $H\boxtimes P\boxtimes K_\ell$  for some planar graph $H$ with treewidth 3 and for some path $P$, where $\ell=\max\{2g\floor{\frac{d}{2}},d+3\floor{\frac{d}{2}}-3\}$. 
\end{thm}

Framed graphs (for $g=0$) were introduced by \citet{BDGGMR20} and are useful because they include several interesting graph classes, as shown by the following three examples.

First, every graph with Euler genus $g$ is a subgraph of a $(g,3)$-framed multigraph. Thus \cref{FramedProduct} with $d=3$ implies \cref{SurfaceStructureImproved}.

Now consider map graphs. Start with a graph $G$ embedded in a surface $\Sigma$ without crossings, with each face labelled a `nation' or a `lake', where each vertex of $G$ is incident with at most $d$ nations. Let $M$ be the graph whose vertices are the nations of $G$, where two vertices are adjacent in $G$ if the corresponding faces in $G$ share a vertex. Then $M$ is called a \defn{$(\Sigma,d)$-map graph}. If $\Sigma$ has Euler genus at most $g$, then $M$ is called a \defn{$(g,d)$-map graph}. Graphs embeddable in $\Sigma$ are precisely the $(\Sigma,3)$-map graphs~\citep{DEW17}. So map graphs are a natural generalisation of graphs embeddable in surfaces.

We show that every $(\Sigma,d)$-map graph is a spanning subgraph of $G^{(d)}$ for some multigraph $G$ embedded in $\Sigma$ without crossings; see \cref{MapGraphFramed}. Thus \cref{FramedProduct} implies that $(g,d)$-map graphs have the following product structure.

\begin{thm}\label{MapGraphs}
	Every $(g,d)$-map graph is contained in  $H\boxtimes P\boxtimes K_\ell$ for some planar graph $H$ with treewidth 3 and for some path $P$, where $\ell=\max\{2g\floor{\frac{d}{2}},d+3\floor{\frac{d}{2}}-3\}$. 
\end{thm}

A graph is \defn{$k$-planar} if it has an embedding in the plane where each edge is involved in at most $k$ crossings. This definition has a natural extension for other surfaces $\Sigma$. A graph is \defn{$(\Sigma, k)$-planar} if it has an embedding in $\Sigma$ where each edge is involved in at most $k$ crossings. A graph is \defn{$(g,k)$-planar} if it is $(\Sigma,k)$-planar for some surface $\Sigma$ with Euler genus at most $g$. In the planar setting ($g=0$), these graphs have been extensively studied; see \cite{KLM17,DLM19} for  surveys. 

We show that every $(\Sigma,1)$-planar graph is contained in $G^{(4)}$ for some multigraph $G$ embedded in $\Sigma$ without crossings; see \cref{1PlanarFramed}. Thus \cref{FramedProduct} implies the following product structure theorem.

\begin{thm}\label{1PlanarGraphs}
	Every $(g,1)$-planar graph is contained in $H\boxtimes P\boxtimes K_{\max\{4g,7\}}$ for some planar graph $H$ with treewidth 3 and for some path $P$.
\end{thm}

\citet{DMW} proved that every $(g,k)$-planar graph is contained in $H \boxtimes P \boxtimes K_{\ell}$, for some graph $H$ with treewidth  $\binom{k+4}{3}-1$ where $\ell=\max\{2g,3\}(6k^2+16k+10)$. \citet{HW21b}  improved $\ell$ to $2\max\{2g,3\}(k+1)^2$. In the $k=1$ case, \cref{1PlanarGraphs} is significantly stronger than both these results since $H$ has treewidth $3$ instead of treewidth $9$. As mentioned above, treewidth 3 is best possible, even for planar graphs~\citep{DJMMUW20}. Note that \citet{DMW} previously proved \cref{1PlanarGraphs} in the planar case ($g=0$), and a similar result was independently obtained by \citet{BDHK}.

\section{\large Proofs}

Undefined terms and notation can be found in Diestel's text~\citep{Diestel5}. A \defn{partition} of a graph $G$ is a set $\Pcal$ of non-empty sets of vertices in $G$ such that each vertex of $G$ is in exactly one element of $\Pcal$. Each element of $\Pcal$ is called a \defn{part}. The \defn{quotient} of $\Pcal$ is the graph, denoted by \defn{$G/\Pcal$}, with vertex set $\Pcal$ where distinct parts $A,B\in \mathcal{P}$ are adjacent in $G/\Pcal$ if and only if some vertex in $A$ is adjacent in $G$ to some vertex in $B$. An \defn{$H$-partition} of $G$ is a partition $\Pcal=(A_x:x \in V(H))$ where $H\cong G/\Pcal$. For simplicity, we sometimes abuse notation and say $J \in \Pcal$ where $J$ is a subgraph of $G$ with $V(J) \in \Pcal$. 

If $T$ is a tree rooted at a vertex $r$, then a non-empty path $P$ in $T$ is \defn{vertical} if the vertex of $P$ closest to $r$ in $T$ is an end-vertex of $P$. If $T$ is a rooted spanning tree in a graph $G$, then a \defn{tripod} in $G$ (with respect to $T$) consists of up to three pairwise vertex-disjoint vertical paths in $T$ whose lower end-vertices form a clique in $G$. 

A \defn{layering} of a graph $G$ is an ordered partition $\mathcal{L}:=(L_0,L_1,\dots)$ of $V(G)$ such that for every edge $vw \in E(G)$, if $v \in L_i$ and $w \in L_j$, then $|i-j|\leq 1$. A \defn{layered partition} $(\Pcal,\mathcal{L})$ of a graph $G$ consists of a partition $\Pcal$ and a layering~$\mathcal{L}$ of $G$.  If $\Pcal=(A_x:x\in V(H))$ is an $H$-partition, then $(\Pcal,\mathcal{L})$ is a \defn{layered $H$-partition} with \defn{width} $\max\{|A_x\cap L|:x\in V(H), L \in \mathcal{L}\}$. 
Layered partitions were introduced by \citet{DJMMUW20} who observed the following connection to strong products (which follows directly from the definitions).

\begin{obs}[\cite{DJMMUW20}]\label{VerticalPathsLemma}
	For all graphs $G$ and $H$, 
	$G$ is contained in $H\boxtimes P\boxtimes K_{\ell}$ for some path $P$ if and only if $G$ has a layered $H$-partition $(\Pcal,\mathcal{L})$ with width at most $\ell$. 
\end{obs}

We need the following lemma of \citet{DMW}, which is a special case of their Lemma~24 (which is an extension of Lemma~17 from \citep{DJMMUW20}). 

\begin{lem}[\citep{DMW}]\label{NewDMWv2}
	Let $G^+$ be a plane multigraph in which each face of $G^+$ is bounded by a cycle with length in $\{3,\dots,d\}$. Let $T$ be a spanning tree of $G^+$ rooted at some vertex $r$ on the boundary of the outer-face of $G^+$. Assume there is a vertical path $P$ in $T$ with end-vertices $p_1$ and $p_2$ such that the cycle $C$ obtained from $P$ by adding the edge $p_1p_2$ is a subgraph of $G^+-r$. Let $G$ be the plane graph consisting of all the vertices and edges of $G^+$ contained in $C$ and the interior of $C$. Then $G^{(d)}$ has an $H$-partition $\Pcal$ such that $P\in\Pcal$ and each part $S_i \in \Pcal\setminus\{P\}$ has a partition $\{X_i,Y_i\}$ where $|X_i|\leq d-3$ and $Y_i$ is the union of at most three vertical paths in $T$, and $H$ is planar with treewidth at most 3.
\end{lem}

The next lemma is the heart of our proof.

\begin{lem}
	\label{GenusPartitionNewFrame}
	Let $G$ be a connected multigraph embedded in a surface of Euler genus $g$ without crossings, where each face of $G$ is bounded by a cycle. Then for every spanning tree $T$ of $G$ and every integer $d\geq 3$, $G^{(d)}$ has an $H$-partition $\Pcal$ such that one part $Z\in \Pcal$ is the union of at most $2g$ vertical paths in $T$ and each part $S_i \in \Pcal\setminus\{Z\}$ has a partition $\{X_i,Y_i\}$ where $|X_i|\leq d-3$ and $Y_i$ is the union of at most three vertical paths in $T$, and $H$ is planar with treewidth at most 3. 
\end{lem}

\begin{proof}
	We start by following the proof of \citep[Lemma~21]{DJMMUW20}, which is the heart of the proof of \cref{SurfaceStructure}. Near the end of our proof  we follow a different strategy to obtain the stronger result.
	
	If $g=0$, then the claim follows from \cref{NewDMWv2} 
	by considering an appropriate supergraph $G^{+}$ of $G$. Now assume that $g\geq 1$. Say $G$ has $n$ vertices, $m$ edges, and $f$ faces. By Euler's formula, $n-m+f=2-g$. Let $D$ be the multigraph with vertex-set the set of faces in $G$, where for each edge $e$ of $E(G)\setminus E(T)$, if $f_1$ and $f_2$ are the faces of $G$ with $e$ on their boundary, then there is an edge joining $f_1$ and $f_2$ in $D$. (Think of $D$ as the spanning subgraph of the dual graph consisting of those edges that do not cross edges in $T$.)  Note that $|V(D)|=f=2-g - n + m$ and $|E(D)|=m-(n-1)= |V(D)|-1+g$. Since $T$ is a tree, $D$ is connected; see~\citep[Lemma~11]{DMW17} for a proof. Let $T^*$ be a spanning tree of $D$. Thus $|E(D)\setminus E(T^*)|=g$.  Let $Q=\{a_1b_1,a_2b_2,\dots,a_gb_g\}$ be the set of edges in $G$ dual to the edges in $E(D)\setminus E(T^*)$. Let $r$ be the root of $T$, and for $i\in\{1,2,\dots,g\}$, let $Z_i$ be the union of the $a_ir$-path and the $b_ir$-path in $T$, plus the edge $a_ib_i$. Let $Z := Z_1\cup Z_2\cup\dots\cup Z_g$. By construction, $Z$ is a connected subgraph of $G$; see \cref{ToroidalGridExampleFramed} for an example. In fact, since $r$ is contained in each of the $2g$ vertical paths, $T[V(Z)]$ is connected. Say $Z$ has $p$ vertices and $q$ edges. Since $Z$ consists of a subtree of $T$ plus the $g$ edges in $Q$, we have $q = p-1+g$.
	
	\begin{figure}[!h]
		\centering
		\includegraphics{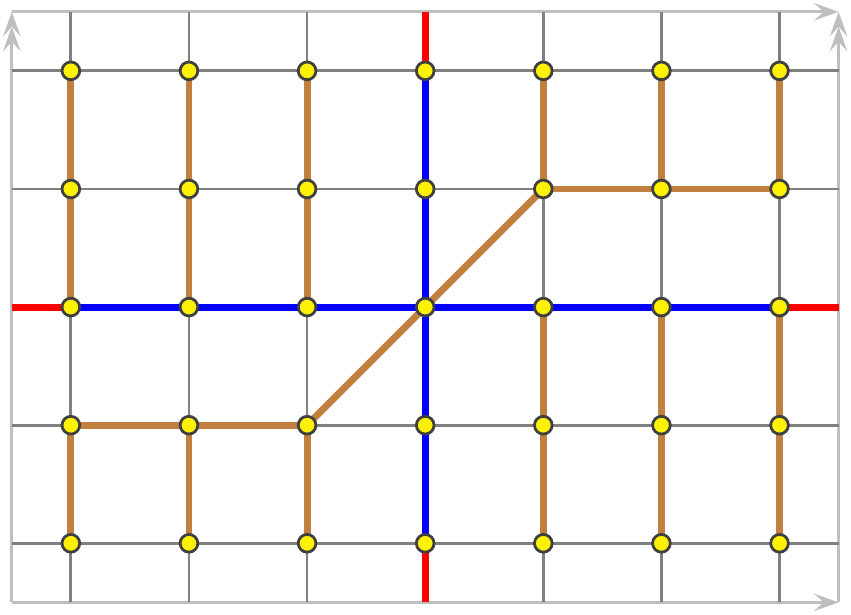}
		\caption{Example of the construction in the proof of \cref{GenusPartitionNewFrame}, where brown edges are in $T$, red edges are in $Q$, and blue edges are in $T$ and in $Z-E(Q)$.}
		\label{ToroidalGridExampleFramed}
	\end{figure}
	
	We now describe how to `cut' along the edges of $Z$ to obtain a new embedded graph $\tilde{G}$; see \cref{CuttingFramed}. First, each edge $e$ of $Z$ is replaced by two edges $e'$ and $e''$ in $\tilde{G}$. Each vertex of $G$ that is not contained in $V(Z)$ is untouched. Consider a vertex $v\in V(Z)$ incident with edges $e_1,e_2,\dots,e_d$ in $Z$ in clockwise order. In $\tilde{G}$ replace $v$ by new vertices $v_1,v_2,\dots,v_d$, where $v_i$ is incident with $e'_i$, $e''_{i+1}$ and all the edges incident with $v$ clockwise from $e_i$ to $e_{i+1}$ (exclusive). Here $e_{d+1}$ means $e_1$ and $e''_{d+1}$ means $e''_1$. This operation defines a cyclic ordering of the edges in $\tilde{G}$ incident with each vertex (where $e''_{i+1}$ is followed by $e'_i$ in the cyclic order at $v_i$). This in turn defines an embedding of $\tilde{G}$ in some orientable surface\footnote{If $G$ is embedded in a non-orientable surface, then the edge signatures for $G$ are ignored in the embedding of $\tilde{G}$.}. Let $Z'$ be the set of vertices introduced in $\tilde{G}$ by cutting through vertices in $Z$. 
	
	\begin{figure}[!t]
		\centering
		\includegraphics{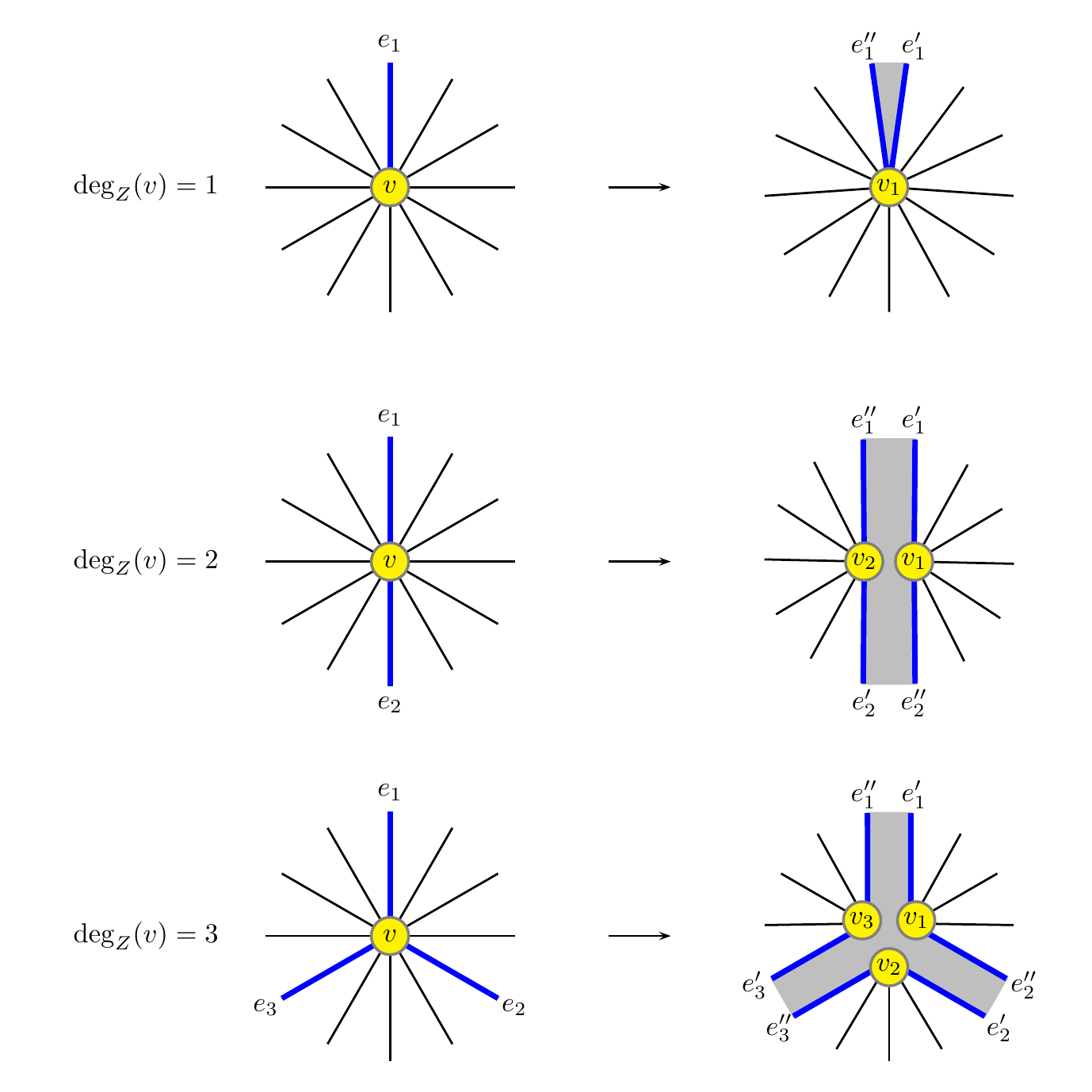}
		\caption{Cutting the blue edges in $Z$ at each vertex.\label{CuttingFramed}}
	\end{figure}
	
	We now show that $\tilde{G}$ is connected. Consider vertices $x_1$ and $x_2$ of $\tilde{G}$. Select faces $f_1$ and $f_2$ of $\tilde{G}$ respectively incident to $x_1$ and $x_2$ that are also faces of $G$. Let $P$ be a path joining $f_1$ and $f_2$ in the dual tree $T^*$. Then the edges of $G$ dual to the edges in $P$ were not split in the construction of $\tilde{G}$. Therefore an $x_1x_2$-walk in $\tilde{G}$ can be obtained by following the boundaries of the faces corresponding to vertices in $P$. Hence $\tilde{G}$ is connected. 
	
	Say $\tilde{G}$ has $n'$ vertices and $m'$ edges, and the embedding of $\tilde{G}$ has $f'$ faces and Euler genus $g'$. Each vertex with degree $d$ in $Z$ is replaced by $d$ vertices in $\tilde{G}$. Each edge in $Z$ is replaced by two edges in $\tilde{G}$, while each edge of $E(G)-E(Z)$ is maintained in $\tilde{G}$. Thus
	$$n' = n -p + \sum_{v\in V(Z)}\deg_Z(v) = n + 2q -p  = n + 2(p-1+g) - p = n +p - 2 + 2g$$ and $m' = m + q = m + p-1 + g$. Each face of $G$ is preserved in $\tilde{G}$. Say $s$ new faces are created by the cutting. Thus $f'=f+s$. Since $\tilde{G}$ is connected, $n'-m'+f'=2-g'$ by Euler's formula. Thus $(n +p - 2 + 2g) - (m + p-1 + g) + (f+s) = 2-g'$, implying $(n-m+f)   - 1 + g + s = 2-g'$. Hence $(2-g)   - 1 + g + s = 2-g'$, implying $g' = 1-s$. Since $g'\geq 0$, we have $s\leq 1$.
	Since $g\geq 1$, by construction, $s\geq 1$. Thus $s=1$ and $g'=0$. Hence  $\tilde{G}$ is plane and all the vertices in $Z'$ are on the boundary of a single face, $F$, of $\tilde{G}$. Moreover, the boundary of $F$ is a cycle $C_F$ and $V(C_F)=Z'$. Consider $F$ to be the outer-face of $\tilde{G}$. 
	
	Now construct a supergraph $G^{+}$ of $\tilde{G}$ by adding a vertex $r^{+}$ in $F$ and edges from $r^{+}$ to each vertex in $Z'$. Then $G^+$ is a plane multigraph where each face of $G^+$ is bounded by a cycle. 
	
	We now depart from the proof of \citet[Lemma~21]{DJMMUW20}. Let $P^{+}$ be an arbitrary path such that $V(P^{+})=V(C_F)$ and let $v^{+}\in V(P^{+})$ be an end-vertex of $P^{+}$. Let $T^{+}$ be the following spanning tree of $G^{+}$ rooted at $r^{+}$. Initialise $T^{+}$ to be the path $P^{+}$ plus the edge $r^{+}v^{+}$. Let $E':=\{vw\in E(T) : v\in Z, w\in V(G)\setminus V(Z)\}$ and $h:=|E'|$. Observe that $T-V(Z)$ is a forest with $h$ components. For each edge $vw\in E'$, $w$ is adjacent to exactly one vertex $v_i\in V(Z')$ introduced when cutting $v$. Add the edge $v_iw$ to $T^{+}$. Finally, add the induced forest $T-V(Z)$ to $T^{+}$; see \cref{SpanningTreePlusFramed}. Then $T^{+}$ is connected since each component of $T-V(Z)$ is adjacent in $T^{+}$ to some vertex in $V(P^+)$. Furthermore, since $|V(T^+)|=|V(P^+)|+|V(G)\setminus V(Z)|$ and $|E(T^+)|=|E(P^+)|+h+(|V(G)\setminus V(Z)|-h)=|V(P^+)|+|V(G)\setminus V(Z)|-1$, it follows that $T^+$ is indeed a spanning tree of $G^+$. Consider each component of $T-V(Z)$ to be a subtree of $T^+$.
	
	\begin{figure}[!h]
		\centering
		\includegraphics[width=0.55\textwidth]{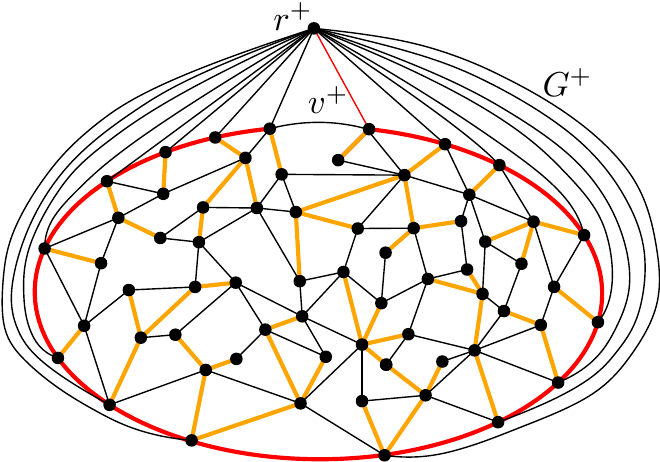}
		\caption{Example of the spanning tree $T^+$ in the graph $G^+$, where the edges in $E(P^{+})\cup \{r^{+}v^{+}\}$ are red and the edges that are either in $E(T-V(Z))$ or of the form $v_iw$ are orange.} \label{SpanningTreePlusFramed}
	\end{figure}
	
	Now every vertical path in $T^{+}$ contained in $V(G)\bs V(Z)$ corresponds to a vertical path in $T$. Every maximal vertical path in $T^{+}$ consists of the edge $r^{+}v^{+}$, a subpath of $P^{+}$, some edge $v_iw$ (where $w\in V(G)\bs V(Z)$), followed by a path in $T-V(Z)$ from $w$ to a leaf in $T$. Since every vertical path $P$ in $T^{+}$ is contained in some maximal vertical path in $T^{+}$, it follows that $P \cap (V (G) \bs V (Z))$ is a vertical path in $T$. Thus every vertical path in $T^+$ that is contained in $V(G)\bs V(Z)$ is a vertical path in $T$.
	
	Triangulate every face in $G^{+}$ whose facial cycle has length greater than $d$. Since $r^+$ is on the boundary of the outer-face of $G^{+}$, $V(P^+)=V(C_F)$, every facial cycle has length in $\{3,\dots,d\}$ and $P^+$ is a vertical path of $T^+$, \cref{NewDMWv2} is applicable. Let $\mathcal{P'}$ be the $H$-partition of $\tilde{G}^{(d)}$ given by \cref{NewDMWv2}. Therefore, $H$ is planar with treewidth at most $3$, where $P^+\in\mathcal{P'}$ and each part in $S_i \in \mathcal{P'}\setminus\{P^+\}$ has a partition $\{X_i,Y_i\}$ where $|X_i|\leq d-3$ and $Y_i$ is the union of at most three vertical paths in $T'$. Let $\mathcal{P}$ be the partition of $G^{(d)}$ obtained by replacing $P^+$ by $Z$. Since $V(P^+)=V(Z')$ and all the split vertices of $G$ are in $Z$, we have $G^{(d)}/\mathcal{P}\cong \tilde{G}^{(d)}/\mathcal{P'} \cong H$. Hence $\mathcal{P}$ is also an $H$-partition where $H$ is planar with treewidth at most $3$. In addition, since each vertical path in $T^{+}$ that is disjoint from $V(Z') \cup \{r^+\}$ is a vertical path in $T$, each part $S_i \in \Pcal \bs \{Z\}$ has a partition $\{X_i,Y_i\}$ where $|X_i|\leq d-3$ and $Y_i$ is the union of at most three vertical paths in $T$, as required.
\end{proof}

\cref{FramedProduct} is an immediate consequence of \cref{VerticalPathsLemma} and the next lemma. 

\begin{lem}\label{FramedPartition}
	Let $G$ be a multigraph embedded in a surface of Euler genus $g$ without crossings, where each face is bounded by a cycle. Then $G^{(d)}$ has a layered $H$-partition $(\Pcal,\Lcal)$ with width at most $\max\{2g\floor{\frac{d}{2}},d+3\floor{\frac{d}{2}}-3\}$, such that $H$ is planar with treewidth at most 3.
\end{lem}

\begin{proof}
	Since each face of $G$ is bounded by a cycle, $G$ is connected. Let $T$ be a BFS-spanning tree of $G$ with corresponding BFS-layering\footnote{If $G$ is a connected graph and $T$ is a spanning tree of $G$ rooted at vertex $r$, then $T$ is \defn{BFS} if  $\dist_T(v,r)=\dist_G(v,r)$ for every $v\in V(G)$. A  layering $(L_0,L_1,\dots)$ of a graph $G$ is \defn{BFS} if $L_0 = \{r\}$ for some root vertex $r\in V(G)$ and $L_i=\{v\in V(G):\dist_G(v,r)=i\}$ for all $i\geq 1$.} $(V_0,V_1,\dots)$. By \cref{GenusPartitionNewFrame}, $G^{(d)}$ has an $H$-partition $\Pcal$ such that one part $Z\in \Pcal$ is the union of at most $2g$ vertical paths in $T$ and each part $S_i\in \Pcal\bs \{Z\}$ has a partition $\{X_i,Y_i\}$ where $|X_i|\leq d-3$ and $Y_i$ is the union of at most three vertical paths in $T$, and $H$ is planar with treewidth at most $3$. It remains to adjust the layering of $G$ to obtain a layering of $G^{(d)}$. If $uv\in E(G^{(d)})$ then $\dist_G(u,v)\leq \floor{\frac{d}{2}}$. Thus if $u\in V_i$ and $v\in V_j$ then $|i-j|\leq \floor{\frac{d}{2}}$. For each $j\in \mathbb{N},$ let $L_j=V_{j\floor{\frac{d}{2}}}\cup \dots \cup V_{(j+1)\floor{\frac{d}{2}}-1}$. Then $(\Pcal,\Lcal=(L_0,L_1,\dots))$ is a layered $H$-partition of $G^{(d)}$ with width at most $\max\{2g\floor{\frac{d}{2}},d+3\floor{\frac{d}{2}}-3\}$, as required.
\end{proof}

We conclude by showing that $(\Sigma,d)$-map graphs and $(\Sigma,1)$-planar graphs are contained in framed graphs. \citet{DMW} proved the following result in the case of plane map graphs (and similar results were previously known in the literature~\citep{CGP06,Brandenburg19,Brandenburg20}). An analogous proof works for arbitrary surfaces, which we include for completeness. Together with \cref{FramedProduct}, this implies \cref{MapGraphs}.

\begin{lem}
	\label{MapGraphFramed}
	For every surface $\Sigma$ and integer $d\geq 3$, every $(\Sigma,d)$-map graph is a subgraph of $G^{(d)}$ for some multigraph $G$ embedded in $\Sigma$ without crossings, where each face of $G$ is bounded by a cycle. 
\end{lem}

\begin{proof}
	Let $G_0$ be a graph embedded in $\Sigma$, with each face labelled a nation or a lake, and where each vertex of $G_0$ is incident with at most $d$ nations. Let $M$ be the corresponding map graph. 
	
	If $G_0$ has a face $F$ of length 2, then add a new vertex inside $F$ adjacent to both vertices on the boundary of $F$, which creates two new triangular faces $F_1$ and $F_2$. If $F$ is a lake, then make $F_1$ and $F_2$ lakes. If $F$ is a nation, then make $F_1$ a nation and make $F_2$ a lake. The resulting map graph is still $M$. So we may assume that $G_0$ is an edge-maximal multigraph embedded in $\Sigma$ with no face of length 2 (and with each face labelled a nation or a lake), such that $M$ is the corresponding map graph. This is well-defined since the assumption of having no face of length 2 implies that $|E(G_0)|\leq 3(|V(G)|+g-2)$, where $g$ is the Euler genus of $\Sigma$. 
	
	Suppose that some face $f$ of $G_0$ has a disconnected boundary. Let $v$ and $w$ be vertices in distinct components of the boundary of $f$. Add the edge $vw$ to $G_0$ across $f$. The corresponding map graph is unchanged, which contradicts the edge-maximality of $G_0$. Thus each face of $G_0$ has a connected boundary. Suppose that some face $f$ of $G_0$ has a repeated vertex $v$ in the boundary walk of $f$. Let $u,v,w$ be consecutive vertices on the boundary of $f$. So $u,v,w$ are distinct. Add the edge $uw$ inside $f$ so that $uvw$ bounds a disk. Label the resulting face $uvw$ as a lake. Since $v$ appears elsewhere in the boundary of $f$, the corresponding map graph is unchanged, which contradicts the edge-maximality of $G_0$. Thus no facial walk of $G_0$ has a repeated vertex. Since each facial walk is connected, every face of $G_0$ is bounded by a cycle.
	
	Let $G_0^*$ be the dual multigraph of $G_0$. So the vertices of $G_0^*$ correspond to faces of $G_0$, and each vertex of $G_0^*$ is a nation vertex or a lake vertex. Since every face of $G_0$ is bounded by a cycle, every face of $G^*_0$ is bounded by a cycle. 
	
	Let $x$ be a vertex of $G_0$, let $F_x$ be the corresponding face of $G_0^*$, and let $(v_1,\ldots,v_s)$ be the facial cycle of $F_x$.  Let $C_x:=(w_1,\ldots,w_r)$ be the circular subsequence of $(v_1,\ldots,v_s)$ consisting of only the nation vertices. Since $x$ is incident to at most $d$ nations, $r\le d$. 
	
	Let $G$ be the supergraph of $G_0^*$ obtained by adding an edge between each pair of consecutive vertices in $C_x=(w_1,\ldots,w_r)$ for each vertex $x$ of $G_0$. 
	The graph consisting of $C_x$ plus these added edges is called the \defn{nation cycle} (of $x$). Note that if $r=1$ then the nation cycle has no edges, and if $r=2$ then the nation cycle has one edge. Since every face of $G_0^*$ is bounded by a cycle, every face of $G$ is bounded by a cycle. Moreover, each nation cycle of length at least $3$ is now a facial cycle of $G$ with length at most $d$. By construction, $G$ embeds in $\Sigma$ with no crossings. Let $G^{(d)}$ be the $d$-framed graph whose frame is $G$.
	
	By definition, $V(M) \subseteq V(G^{(d)})$. To prove the claim, it suffices to show that $E(M)\subseteq E(G^{(d)})$.  Indeed, if $vw\in E(M)$ then the nation faces corresponding to $v$ and $w$ have a common vertex $x$ on their boundary. The vertex $x$ corresponds to a face $F_x$ in $G_0^*$ and the facial cycle of $F_x$ contains $v$ and $w$.  Therefore, the nation cycle $C_x$ of $F_x$ contains $v$ and $w$. If $C_x$ has length $2$ then $vw\in E(G)\subseteq E(G^{(d)})$. If $C_x$ has length at least $3$ then it has length at most $d$ and it bounds a face in $G$. So $vw\in E(G^{(d)})$.
\end{proof}

\citet{DMW} proved the following result in the case of 1-planar graphs (and similar results were previously known in the literature~\citep{CGP06,BDGGMR20,Brandenburg19,Brandenburg20}). An analogous proof works for arbitrary surfaces, which we include for completeness. Together with \cref{FramedProduct}, this implies \cref{1PlanarGraphs}.

\begin{lem}
	\label{1PlanarFramed}
	Every $(\Sigma,1)$-planar graph $G$ with at least three vertices is contained in $G_0^{(4)}$ for some multigraph $G_0$ embedded in $\Sigma$ with no crossings where each face of $G_0$ is bounded by a cycle.
\end{lem}

\begin{proof}
	We may assume that $G$ is embedded in $\Sigma$ with at most one crossing on each edge, such that no two edges of $G$ incident to a common vertex cross, since such a crossing can be removed by a local modification to obtain an embedding of $G$ in which the two edges do not cross.
	
	Initialise $G':=G$. Add edges to $G'$ to obtain an  edge-maximal multigraph embedded in $\Sigma$ such that each edge is in at most one crossing, no two edges incident to a common vertex cross, and no face is bounded by two parallel edges. The final condition ensures that $G'$ is well-defined, since it follows from Euler's formula that if $G$ has $k$ crossings, then $|E(G')| \leq 3( |V(G)| +k+g-2) -2k$.
	
	Consider crossing edges $e_1=vw$ and $e_2=xy$ in $G'$. So $v,w,x,y$ are distinct. Since $e_1$ is the only edge that crosses $e_2$ and $e_2$ is the only edge that crosses $e_1$, by the edge-maximality of $G'$, there is a cycle $C=(v,x,w,y)$ in $G'$ that bounds a disc whose interior intersects no edge of $G'$ except $e_1$ and $e_2$. 
	
	Let $G_0$ be the embedded multigraph obtained from $G'$ by deleting each pair of crossing edges. Thus the above-defined cycle $C$ bounds a face of $G_0$. By the edge-maximality of $G'$, every other face of $G_0$ (that is, not arising from a pair of deleted crossing edges) is a triangular face of $G'$.  Thus, $G_0$ is a multigraph embedded in $\Sigma$ with no crossings, such that each face of $G_0$ is bounded by a 3-cycle or a 4-cycle, and $G$ is contained in $G_0^{(4)}$.
\end{proof}

	\acknowledgements
	\label{sec:ack}
	This research was initiated at the workshop on Graph Product Structure Theory held at the Banff International Research Station in November 2021. Thanks to the referee for several helpful suggestions.
	
%	\bibliographystyle{abbrvnat}
	%overleaf 	\bibliography{myBibliography}
	%david's computer \bibliography{../../../bibtex/myBibliography}

\begin{thebibliography}{23}
		\providecommand{\natexlab}[1]{#1}
		\providecommand{\url}[1]{\texttt{#1}}
		\expandafter\ifx\csname urlstyle\endcsname\relax
		\providecommand{\doi}[1]{doi: #1}\else
		\providecommand{\doi}{doi: \begingroup \urlstyle{rm}\Url}\fi
		
		\bibitem[Bekos et~al.(2020)Bekos, Lozzo, Griesbach, Gronemann, Montecchiani,
		and Raftopoulou]{BDGGMR20}
		M.~A. Bekos, G.~D. Lozzo, S.~Griesbach, M.~Gronemann, F.~Montecchiani, and
		C.~N. Raftopoulou.
		\newblock Book embeddings of nonplanar graphs with small faces in few pages.
		\newblock In S.~Cabello and D.~Z. Chen, editors, \emph{Proc. 36th Int'l Symp.
			on Computational Geometry \textup{(SoCG '20)}}, volume 164 of \emph{LIPIcs},
		pages 16:1--16:17. Schloss Dagstuhl, 2020.
		\newblock \doi{10.4230/LIPIcs.SoCG.2020.16}.
		
		\bibitem[Bekos et~al.(2022)Bekos, Da~Lozzo, Hlin\u{e}n\'{y}, and
		Kaufmann]{BDHK}
		M.~A. Bekos, G.~Da~Lozzo, P.~Hlin\u{e}n\'{y}, and M.~Kaufmann.
		\newblock Graph product structure for $h$-framed graphs, 2022.
		\newblock arXiv:2204.11495.
		
		\bibitem[Bonamy et~al.(2020)Bonamy, Gavoille, and Pilipczuk]{BGP20}
		M.~Bonamy, C.~Gavoille, and M.~Pilipczuk.
		\newblock Shorter labeling schemes for planar graphs.
		\newblock In S.~Chawla, editor, \emph{Proc. ACM-SIAM Symp. on Discrete
			Algorithms \textup{(SODA '20)}}, pages 446--462, 2020.
		\newblock \doi{10.1137/1.9781611975994.27}.
		
		\bibitem[Bonnet et~al.(2022)Bonnet, Kwon, and Wood]{BKW}
		E.~Bonnet, O.~Kwon, and D.~R. Wood.
		\newblock Reduced bandwidth: a qualitative strengthening of twin-width in
		minor-closed classes (and beyond), 2022.
		\newblock arXiv:2202.11858.
		
		\bibitem[Bose et~al.(2020)Bose, Dujmovi\'c, Javarsineh, and Morin]{BDJM}
		P.~Bose, V.~Dujmovi\'c, M.~Javarsineh, and P.~Morin.
		\newblock Asymptotically optimal vertex ranking of planar graphs, 2020.
		\newblock arXiv:2007.06455.
		
		\bibitem[Brandenburg(2019)]{Brandenburg19}
		F.~J. Brandenburg.
		\newblock Characterizing and recognizing 4-map graphs.
		\newblock \emph{Algorithmica}, 81\penalty0 (5):\penalty0 1818--1843, 2019.
		\newblock \doi{10.1007/s00453-018-0510-x}.
		
		\bibitem[Brandenburg(2020)]{Brandenburg20}
		F.~J. Brandenburg.
		\newblock Book embeddings of $k$-map graphs, 2020.
		\newblock arXiv:2012.06874.
		
		\bibitem[Chen et~al.(2006)Chen, Grigni, and Papadimitriou]{CGP06}
		Z.-Z. Chen, M.~Grigni, and C.~H. Papadimitriou.
		\newblock Recognizing hole-free 4-map graphs in cubic time.
		\newblock \emph{Algorithmica}, 45\penalty0 (2):\penalty0 227--262, 2006.
		\newblock \doi{10.1007/s00453-005-1184-8}.
		
		\bibitem[Didimo et~al.(2019)Didimo, Liotta, and Montecchiani]{DLM19}
		W.~Didimo, G.~Liotta, and F.~Montecchiani.
		\newblock A survey on graph drawing beyond planarity.
		\newblock \emph{{ACM} Comput. Surv.}, 52\penalty0 (1):\penalty0 4:1--4:37,
		2019.
		\newblock \doi{10.1145/3301281}.
		
		\bibitem[Diestel(2018)]{Diestel5}
		R.~Diestel.
		\newblock \emph{Graph theory}, volume 173 of \emph{Graduate Texts in
			Mathematics}.
		\newblock Springer, 5th edition, 2018.
		
		\bibitem[D\k{e}bski et~al.(2021)D\k{e}bski, Felsner, Micek, and
		Schr\"{o}der]{DFMS21}
		M.~D\k{e}bski, S.~Felsner, P.~Micek, and F.~Schr\"{o}der.
		\newblock Improved bounds for centered colorings.
		\newblock \emph{Adv. Comb.}, \#8, 2021.
		\newblock \doi{10.19086/aic.27351}.
		
		\bibitem[Dujmovi\'c et~al.(2017)Dujmovi\'c, Eppstein, and Wood]{DEW17}
		V.~Dujmovi\'c, D.~Eppstein, and D.~R. Wood.
		\newblock Structure of graphs with locally restricted crossings.
		\newblock \emph{SIAM J. Discrete Math.}, 31\penalty0 (2):\penalty0 805--824,
		2017.
		\newblock \doi{10.1137/16M1062879}.
		
		\bibitem[Dujmovi{\'c} et~al.(2017)Dujmovi{\'c}, Morin, and Wood]{DMW17}
		V.~Dujmovi{\'c}, P.~Morin, and D.~R. Wood.
		\newblock Layered separators in minor-closed graph classes with applications.
		\newblock \emph{J. Combin. Theory Ser. B}, 127:\penalty0 111--147, 2017.
		\newblock \doi{10.1016/j.jctb.2017.05.006}.
		
		\bibitem[Dujmovi{\'c} et~al.(2019)Dujmovi{\'c}, Morin, and Wood]{DMW}
		V.~Dujmovi{\'c}, P.~Morin, and D.~R. Wood.
		\newblock Graph product structure for non-minor-closed classes, 2019.
		\newblock arXiv:1907.05168.
		
		\bibitem[Dujmovi{\'c} et~al.(2020{\natexlab{a}})Dujmovi{\'c}, Esperet, Joret,
		Walczak, and Wood]{DEJWW20}
		V.~Dujmovi{\'c}, L.~Esperet, G.~Joret, B.~Walczak, and D.~R. Wood.
		\newblock Planar graphs have bounded nonrepetitive chromatic number.
		\newblock \emph{Adv. Comb.}, \#5, 2020{\natexlab{a}}.
		\newblock \doi{10.19086/aic.12100}.
		
		\bibitem[Dujmovi{\'c} et~al.(2020{\natexlab{b}})Dujmovi{\'c}, Joret, Micek,
		Morin, Ueckerdt, and Wood]{DJMMUW20}
		V.~Dujmovi{\'c}, G.~Joret, P.~Micek, P.~Morin, T.~Ueckerdt, and D.~R. Wood.
		\newblock Planar graphs have bounded queue-number.
		\newblock \emph{J. ACM}, 67\penalty0 (4):\penalty0 \#22, 2020{\natexlab{b}}.
		\newblock \doi{10.1145/3385731}.
		
		\bibitem[Dujmovi\'c et~al.(2021)Dujmovi\'c, Esperet, Gavoille, Joret, Micek,
		and Morin]{DEJGMM21}
		V.~Dujmovi\'c, L.~Esperet, C.~Gavoille, G.~Joret, P.~Micek, and P.~Morin.
		\newblock Adjacency labelling for planar graphs (and beyond).
		\newblock \emph{J. ACM}, 68\penalty0 (6):\penalty0 42, 2021.
		\newblock \doi{10.1145/3477542}.
		
		\bibitem[Dujmovi{\'c} et~al.(2022)Dujmovi{\'c}, Esperet, Morin, Walczak, and
		Wood]{DEMWW22}
		V.~Dujmovi{\'c}, L.~Esperet, P.~Morin, B.~Walczak, and D.~R. Wood.
		\newblock Clustered 3-colouring graphs of bounded degree.
		\newblock \emph{Combin. Probab. Comput.}, 31\penalty0 (1):\penalty0 123--135,
		2022.
		\newblock \doi{10.1017/S0963548321000213}.
		
		\bibitem[Esperet et~al.(2020)Esperet, Joret, and Morin]{EJM}
		L.~Esperet, G.~Joret, and P.~Morin.
		\newblock Sparse universal graphs for planarity, 2020.
		\newblock arXiv:2010.05779.
		
		\bibitem[Hickingbotham and Wood(2021)]{HW21b}
		R.~Hickingbotham and D.~R. Wood.
		\newblock Shallow minors, graph products and beyond planar graphs, 2021.
		\newblock arXiv:2111.12412.
		
		\bibitem[Huynh et~al.(2021)Huynh, Mohar, {\v{S}}{\'a}mal, Thomassen, and
		Wood]{HMSTW}
		T.~Huynh, B.~Mohar, R.~{\v{S}}{\'a}mal, C.~Thomassen, and D.~R. Wood.
		\newblock Universality in minor-closed graph classes, 2021.
		\newblock arXiv:2109.00327.
		
		\bibitem[Kobourov et~al.(2017)Kobourov, Liotta, and Montecchiani]{KLM17}
		S.~G. Kobourov, G.~Liotta, and F.~Montecchiani.
		\newblock An annotated bibliography on 1-planarity.
		\newblock \emph{Comput. Sci. Rev.}, 25:\penalty0 49--67, 2017.
		\newblock \doi{10.1016/j.cosrev.2017.06.002}.
		
		\bibitem[Mohar and Thomassen(2001)]{MoharThom}
		B.~Mohar and C.~Thomassen.
		\newblock \emph{Graphs on surfaces}.
		\newblock Johns Hopkins University Press, 2001.
		
	\end{thebibliography}
	
	\label{sec:biblio}
	
	\def\soft#1{\leavevmode\setbox0=\hbox{h}\dimen7=\ht0\advance \dimen7
		by-1ex\relax\if t#1\relax\rlap{\raise.6\dimen7
			\hbox{\kern.3ex\char'47}}#1\relax\else\if T#1\relax
		\rlap{\raise.5\dimen7\hbox{\kern1.3ex\char'47}}#1\relax \else\if
		d#1\relax\rlap{\raise.5\dimen7\hbox{\kern.9ex \char'47}}#1\relax\else\if
		D#1\relax\rlap{\raise.5\dimen7 \hbox{\kern1.4ex\char'47}}#1\relax\else\if
		l#1\relax \rlap{\raise.5\dimen7\hbox{\kern.4ex\char'47}}#1\relax \else\if
		L#1\relax\rlap{\raise.5\dimen7\hbox{\kern.7ex
				\char'47}}#1\relax\else\message{accent \string\soft \space #1 not
			defined!}#1\relax\fi\fi\fi\fi\fi\fi}
	
\end{document}